\documentclass[a4paper,12pt]{article}

\usepackage[T2A]{fontenc}
\usepackage[cp1251]{inputenc}
\usepackage{amsfonts,amsmath,amsthm,amscd,amssymb,latexsym,amsbsy,pb-diagram,cite,caption2}
\usepackage[dvips]{graphicx}
\usepackage[mathscr]{eucal}

\newtheorem{definition}{Definition}[section]
\newtheorem{theorem}[definition]{Theorem}
\newtheorem{lemma}[definition]{Lemma}
\newtheorem{proposition}[definition]{Proposition}
\newtheorem{corollary}[definition]{Corollary}

\theoremstyle{definition}
\newtheorem{remark}[definition]{Remark}
\newtheorem{example}[definition]{Example}

\numberwithin{equation}{section}

 \DeclareMathOperator{\card}{card}

\begin{document}
\begin{center}
{\Large\bf Ultrametricity and metric betweenness in tangent spaces
to metric spaces}
\end{center}

\bigskip
\begin{center}
{\bf O. Dovgoshey and D. Dordovskyi}
\end{center}

\begin{abstract}
The paper deals with pretangent spaces to general metric spaces. An
ultrametricity criterion for pretangent spaces is found and it is
closely related to the metric betweenness in the pretangent spaces.
\end{abstract}

\bigskip
{\bf Mathematics Subject Classification (2000):} 54E35.

\bigskip
{\bf Key words:} Metric spaces; Tangent spaces; Ultrametric
spaces; Metric betweenness.

\section{Introduction}
\normalsize

Analysis on metric spaces with no a priori smooth structure is in
need of some generalized differentiations. Important examples of
such generalizations and even an axiomatics  of so-called
"pseudo-gradients" can be found in \cite{Am1,AK1,AK2,Ch,Haj,HeKo,Sh}
and, respectively, in \cite{Am2}. A linear structure, and so a
differentiation, for separable metric spaces can be obtained via
their isometric embeddings into dual spaces of separable Banach
spaces. For the application of this approach to develop a rather
complete theory of rectifiable sets and currents on metric spaces
see \cite{AK1,AK2}. Another natural way to obtain suitable
differentiations on metric spaces is to induce some tangents at the
points of these space. The Gromov--Hausdorff convergence and the
ultra-convergence are, probably, the most widely applied today's
tools for the construction of such tangent spaces (see, for example,
\cite{BBI,BGP} and, respectively, \cite{BrHa,Lyt}). Recently a new
approach to the introduction of the tangent spaces at the points of
general metric spaces was proposed in \cite{DM2}. Our paper is
devoted to the study of the last tangent spaces. We find necessary
and sufficient conditions under which the tangent spaces are
ultrametric, see Theorem \ref{2:T2.10} below. Our second main result
is Theorem \ref{3:T3.2} that completely describes metric spaces for
which tangents are, roughly speaking, the snowflaked versions of
subsets of $\mathbb{R}$. For convenience we recall the main notions
from \cite{DM2}, see also \cite{Dov}.

Let $(X,d)$ be a metric space. Fix a sequence  $\tilde{r}$ of
positive real numbers $r_n$ which tend to zero. In what follows
this sequence $\tilde{r}$ will be called a {\it normalizing
sequence}. Let us denote by $\tilde{X}$ the set of all sequence of
points from $X$.

\begin{definition}\label{1:d1.1}
Two sequence $\tilde{x},\tilde{y}\in\tilde{X}$,
$\tilde{x}=\{x_n\}_{n\in\mathbb{N}}$ and
$\tilde{y}=\{y_n\}_{n\in\mathbb{N}}$ are mutually stable with
respect to (w.r.t.) a normalizing sequence
$\tilde{r}=\{r_n\}_{n\in\mathbb{N}}$ if there is a finite limit
\begin{equation}\label{1:eq1.1}
\underset{n\rightarrow\infty}{\lim}\frac{d(x_n,y_n)}{r_n}=:
\tilde{d}_{\tilde{r}}(\tilde{x},\tilde{y})=
\tilde{d}(\tilde{x},\tilde{y}).
\end{equation}
\end{definition}

We shall say that a family $\tilde{F}\subseteq\tilde{X}$ is {\it
self-stable} (w.r.t. $\tilde{r}$) if every two
$\tilde{x},\tilde{y}\in\tilde{F}$ are mutually stable. A family
$\tilde{F}\subseteq\tilde{X}$ is {\it maximal self-stable} if
$\tilde{F}$ is self-stable and for an arbitrary
$\tilde{z}\in\tilde{X}\setminus\tilde{F}$ there is
$\tilde{x}\in\tilde{F}$ such that $\tilde{x}$ and $\tilde{z}$ are
not mutually stable.

A standard application of Zorn's lemma leads to the following

\begin{proposition}\label{1:p1.2}
Let $(X,d)$ be a metric space and let $p\in X$. Then for every
normalizing sequence $\tilde{r}=\{r_n\}_{n\in\mathbb{N}}$ there
exists a maximal self-stable family $\tilde{X}_{p,\tilde{r}}$ such
that $\tilde{p}=\{p,p,...\}\in\tilde{X}_{p,\tilde{r}}$.
\end{proposition}

Consider a function
$\tilde{d}:\tilde{X}_{p,\tilde{r}}\times\tilde{X}_{p,\tilde{r}}\rightarrow\mathbb{R}$
where
$\tilde{d}(\tilde{x},\tilde{y})=\tilde{d}_{\tilde{r}}(\tilde{x},\tilde{y})$
is defined by \eqref{1:eq1.1}. Obviously, $\tilde{d}$ is symmetric
and nonnegative. Moreover, the triangle inequality implies
$$
\tilde{d}(\tilde{x},\tilde{y})\leq
\tilde{d}(\tilde{x},\tilde{z})+\tilde{d}(\tilde{z},\tilde{y})
$$
for all $\tilde{x},\tilde{y},\tilde{z}$ from
$\tilde{X}_{p,\tilde{r}}$. Hence
$(\tilde{X}_{p,\tilde{r}},\tilde{d})$ is a pseudometric space.

Define a relation $\sim$ on $\tilde{X}_{p,\tilde{r}}$ by
$\tilde{x}\sim\tilde{y}$ if and only if
$\tilde{d}_{\tilde{r}}(\tilde{x},\tilde{y})=0$. Then $\sim$ is an
equivalence relation. Let us denote by
$\Omega_{p,\tilde{r}}^X=\Omega_{p,\tilde{r}}$ the set of
equivalence classes in $\tilde{X}_{p,\tilde{r}}$ under the
equivalence relation $\sim$. If a function $\rho$ is defined on
$\Omega_{p,\tilde{r}}\times\Omega_{p,\tilde{r}}$ by
\begin{equation}\label{1:eq1.2}
\rho(\alpha,\beta)=\tilde{d}(\tilde{x},\tilde{y})
\end{equation}
for $\tilde{x}\in\alpha$ and $\tilde{y}\in\beta$, then $\rho$ is
the well-defined metric on $\Omega_{p,\tilde{r}}$. The metric
identification of $(\tilde{X}_{p,\tilde{r}},\tilde{d})$ is, by
definition, the metric space $(\Omega_{p,\tilde{r}},\rho)$.

\begin{definition}\label{1:d1.3}
The space $(\Omega_{p,\tilde{r}}^X,\rho)$ is pretangent to the
space $X$ at the point $p$ w.r.t. a normalizing sequence
$\tilde{r}$.
\end{definition}

Note that $\Omega_{p,\tilde{r}}\neq\varnothing$ because the
constant sequence $\tilde{p}$ belongs to
$\tilde{X}_{p,\tilde{r}}$, see Proposition \ref{1:p1.2}.

Let $\{n_k\}_{k\in\mathbb{N}}$ be an infinite, strictly increasing
sequence of natural numbers. Let us denote by $\tilde{r}'$ the
subsequence $\{r_{n_k}\}_{k\in\mathbb{N}}$ of the normalizing
sequence $\tilde{r}=\{r_{n}\}_{n\in\mathbb{N}}$ and let
$\tilde{x}':=\{x_{n_k}\}_{k\in\mathbb{N}}$ for every
$\tilde{x}=\{x_n\}_{n\in\mathbb{N}}\in\tilde{X}$. It is clear that
if $\tilde{x}$ and $\tilde{y}$ are mutually stable w.r.t.
$\tilde{r}$, then $\tilde{x}'$ and $\tilde{y}'$ are mutually
stable w.r.t. $\tilde{r}'$ and
\begin{equation}\label{1:eq1.3}
\tilde{d}_{\tilde{r}}(\tilde{x},\tilde{y})=\tilde{d}_{\tilde{r}'}(\tilde{x}',\tilde{y}').
\end{equation}

If $\tilde{X}_{p,\tilde{r}}$ is a maximal self-stable (w.r.t.
$\tilde{r}$) family, then, by Zorn's lemma, there exists a maximal
self-stable (w.r.t. $\tilde{r}'$) family
$\tilde{X}_{p,\tilde{r}'}$ such that
$$
\{\tilde{x}':\tilde{x}\in\tilde{X}_{p,\tilde{r}}\}\subseteq\tilde{X}_{p,\tilde{r}'}.
$$

Denote by $in_{\tilde{r}'}$ the mapping from
$\tilde{X}_{p,\tilde{r}}$ to $\tilde{X}_{p,\tilde{r}'}$ with
$in_{\tilde{r}'}(\tilde{x})=\tilde{x}'$ for all
$\tilde{x}\in\tilde{X}_{p,\tilde{r}}$. It follows from
\eqref{1:eq1.3} that, after the metric identifications,
$in_{\tilde{r}'}$ pass to an isometric embedding
$em':\Omega_{p,\tilde{r}}^X\rightarrow\Omega_{p,\tilde{r}'}^X$
under which the diagram
\begin{equation}\label{1:eq1.4}
\begin{diagram}
\node{\tilde X_{p,\tilde r}}\arrow[2]{e,t}{\text{in}_{\tilde r'}}
                       \arrow[2]{s,l}{\pi}\node[2]{\tilde X_{p,\tilde
                       r'}}
                       \arrow[2]{s,r}{\pi'}\\ \\
\node{\Omega_{p,\tilde
r}^X}\arrow[2]{e,t}{em'}\node[2]{\Omega_{p,\tilde r'}^X}
\end{diagram}
\end{equation}
is commutative. Here $\pi,\pi'$ are canonical projection maps,
$\pi(\tilde{x}):=\{\tilde{y}\in~\tilde{X}_{p,\tilde{r}}:\tilde{d}_{\tilde{r}}(\tilde{x},\tilde{y})=0\}$
and
$\pi'(\tilde{x}):=\{\tilde{y}\in\tilde{X}_{p,\tilde{r}'}:\tilde{d}_{\tilde{r}'}(\tilde{x},\tilde{y})=0\}$.

Let $X$ and $Y$ be two metric spaces. Recall that a map
$f:X\rightarrow Y$ is called an {\it isometry} if $f$ is
distance-preserving and onto.

\begin{definition}\label{1:d1.4}
A pretangent $\Omega_{p,\tilde{r}}^X$ is tangent if
$em':\Omega_{p,\tilde{r}}^X\rightarrow\Omega_{p,\tilde{r}'}^X$ is
an isometry for every $\tilde{X}_{p,\tilde{r}'}$.
\end{definition}

\section{Betweenness exponent and ultrametricity of pretangent spaces}
Let $(X,d)$ be a metric spaces. Denote by $t_0=t_0(X)=t_0(X,d)$
the supremum of positive numbers $t$ for which the function
$(x,y)\mapsto(d(x,y))^t$ is a metric on $X$. It is clear that
$d^{t_0}$ remains a metric if $t_0=t_0(X)<\infty$. The quantity
$t_0(X)$ will be called the {\it betweenness exponent} of the
metric space $(X,d)$. The proofs of the following two lemmas can
be found in \cite{DM1}.

\begin{lemma}\label{2:L2.1}
Let $x,y$ and $z$ be points in a metric space $X$. If the
inequality
\begin{equation}\label{2:eq2.1}
d(x,z)\vee d(z,y)<d(x,y)
\end{equation}
holds, then there exists a unique solution $s_0\in[1,\infty)$ of
the equation
\begin{equation}\label{2:eq2.2}
(d(x,z))^s+(d(z,y))^s=(d(x,y))^s.
\end{equation}
\end{lemma}

For points $x,y$ and $z$ in $X$ write
\begin{equation}\label{2:eq2.3}
s(x,y,z):=
\begin{cases}
s_0 \  \  \  \  \text{ if (\ref{2:eq2.1}) holds},\\
+\infty   \  \  \text{ if (\ref{2:eq2.1}) does not hold}
\end{cases}
\end{equation}
where $s_0$ is the unique root of equation \eqref{2:eq2.2}.

\begin{lemma}\label{2:L2.2}
The equality
$$
t_0(X)=\inf\{s(x,y,z):x,y,z\in X\}
$$
holds for every nonvoid metric space $X$.
\end{lemma}

Recall that a metric space $(X,d)$ is {\it ultrametric} if the
metric $d$ satisfies the {\it ultra-triangle inequality}
$d(x,y)\leqslant d(x,z)\vee d(y,z)$ for all $x,y,z\in~X$.

\begin{remark}\label{2:r2.4}
For ultrametric spaces $X$ we have the equality $t_0(X)=\infty$
since inequality \eqref{2:eq2.1} never holds in these spaces. In
fact $t_0(X)=\infty$ is true if and only if $X$ is ultrametric.
\end{remark}

For every metric space $(Y,d)$ define the subset $Y^{+3}$ of the
Cartesian product $Y\times Y\times Y$ by the rule
\begin{equation}\label{2:eq2.4}
(x,y,z)\in Y^{+3} \Leftrightarrow d(x,z)\geqslant d(x,y)\geqslant
d(y,z)>0.
\end{equation}
It is clear that $Y^{+3}\neq\varnothing$ if and only if $\card
Y\geqslant 3$.

\begin{proposition}\label{2:as2.5}
Let $(X,d)$be a metric space and let $p$ be a limit point of $X$.
If the equality
\begin{equation}\label{2:eq2.5}
\underset{\underset{(x,y,z)\in X^{+3}}{x,y,z\rightarrow
p}}{\lim}s(x,z,y)=s_0\in [1,\infty],
\end{equation}
holds, then for every
$\Omega_{p,\tilde{r}}=\Omega_{p,\tilde{r}}^X$ and all
$(\beta,\gamma,\delta)\in\Omega^{+3}_{p,\tilde{r}}$ we have the
equality
\begin{equation}\label{2:eq2.6}
\rho(\beta,\delta)=\left((\rho(\beta,\gamma))^{s_0}+(\rho(\delta,\gamma))^{s_0}
\right)^\frac{1}{s_0}
\end{equation}
with
$$
\left((\rho(\beta,\gamma))^{s_0}+(\rho(\delta,\gamma))^{s_0}
\right)^\frac{1}{s_0}:=\underset{t\rightarrow\infty}{\lim}
\left((\rho(\beta,\gamma))^t+(\rho(\delta,\gamma))^t
\right)^\frac{1}{t}=\rho(\beta,\gamma)\vee\rho(\delta,\gamma)
$$
if $s_0=\infty$.
\end{proposition}

\begin{proof}
Let $(\beta,\gamma,\delta)\in\Omega^{+3}_{p,\tilde{r}}$ where
$\tilde{r}=\{r_n\}_{n\in\mathbb{N}}$ and let
$\{x_n\}_{n\in\mathbb{N}}\in\beta$,
$\{y_n\}_{n\in\mathbb{N}}\in\gamma$,
$\{z_n\}_{n\in\mathbb{N}}\in\delta$. Let
$(d^n_{[1]},d^n_{[2]},d^n_{[3]})$ be a nonincreasing rearrangement
of the vector $d^n:=(d(x_n,z_n),d(x_n,y_n),d(y_n,z_n))$ i.e., the
vectors $d^n$ and $(d^n_{[1]},d^n_{[2]},d^n_{[3]})$ have the same
components but $d^n_{[1]}\geqslant d^n_{[2]}\geqslant d^n_{[3]}$.
It is easy to see that there is a rearrangement
$(x^*_n,y^*_n,z^*_n)$ of $(x_n,y_n,z_n)\in X^3$ such that
$$
(d^n_{[1]},d^n_{[2]},d^n_{[3]})=(d(x^*_n,z^*_n),d(x^*_n,y^*_n),d(y^*_n,z^*_n)).
$$
The last equality and the relation
$(\beta,\gamma,\delta)\in\Omega^{+3}_{p,\tilde{r}}$ imply
$$
(x^*_n,y^*_n,z^*_n)\in X^{+3}
$$
if $n$ is taken large enough. Indeed, it is sufficient to show
that $d^n_{[3]}>0$ which follows from
$$
\begin{array}{l}
0<\rho(\gamma,\delta)=(\rho(\gamma,\beta)\wedge\rho(\gamma,\delta)\wedge\rho(\beta,\delta))\\
\quad
=\underset{n\rightarrow\infty}{\lim}\frac{1}{r_n}(d(x_n,z_n)\wedge
d(x_n,y_n)\wedge d(y_n,z_n))
=\underset{n\rightarrow\infty}{\lim}\frac{d^n_{[3]}}{r_n}.
\end{array}
$$
Consider first the case $s_0<\infty$. Write
$s_n:=s(x_n^*,z_n^*,y_n^*)$ where the function $s$ is define by
\eqref{2:eq2.3}. Now using \eqref{2:eq2.5} and \eqref{2:eq2.6} we
obtain
\begin{equation}\label{2:eq2.9}
\begin{array}{l}
\rho(\beta,\delta)=\underset{n\rightarrow\infty}{\lim}\frac{1}{r_n}(d(x_n,z_n)\vee
d(x_n,y_n)\vee d(y_n,z_n))
=\underset{n\rightarrow\infty}{\lim}\frac{d(x_n^*,z_n^*)}{r_n}\\
\qquad\quad=\underset{n\rightarrow\infty}{\lim}\left(\left(\frac{d(x_n^*,y_n^*)}{r_n}\right)^{s_n}+
\left(\frac{d(y_n^*,z_n^*)}{r_n}\right)^{s_n}\right)^\frac{1}{s_n}=
\underset{n\rightarrow\infty}{\lim}\left(\left(\frac{d^n_{[2]}}{r_n}\right)^{s_n}+
\left(\frac{d^n_{[3]}}{r_n}\right)^{s_n}\right)^\frac{1}{s_n},
\end{array}
\end{equation}
that implies equality \eqref{2:eq2.6} for $s_0<\infty$.

Suppose now that $s_0=\infty$. Let $M$ be an arbitrary positive
constant. Since the function $f(t)=(a^t+b^t)^{\frac{1}{t}}$,
$a,b\in(0,\infty)$, is strictly decreasing in $t\in(0,\infty)$,
see Remark \ref{2:r2.8} below, and since equality \eqref{2:eq2.5}
holds with $s_0=\infty$, we obtain the inequality
$$
\frac{d(x^*_n,z^*_n)}{r_n}\leqslant
\left(\left(\frac{d(x^*_n,y^*_n)}{r_n}\right)^M+
\left(\frac{d(y^*_n,z^*_n)}{r_n}\right)^M\right)^{\frac{1}{M}}
$$
for sufficiently large $n$. Consequently,
\begin{equation}\label{2:eq2.8}
\begin{array}{l}
\rho(\beta,\delta)=\underset{n\rightarrow\infty}{\lim}\frac{d(x^*_n,z^*_n)}{r_n}\\\qquad\quad\leqslant
\underset{n\rightarrow\infty}{\lim}\left(\left(\frac{d(x^*_n,y^*_n)}{r_n}\right)^M+
\left(\frac{d(y^*_n,z^*_n)}{r_n}\right)^M\right)^{\frac{1}{M}}\leqslant
\left((\rho(\beta,\gamma))^M+(\rho(\gamma,\delta))^M\right)^{\frac{1}{M}}.
\end{array}
\end{equation}

Letting $M\rightarrow\infty$ we have
$\rho(\beta,\delta)\leqslant\rho(\beta,\gamma)\vee\rho(\gamma,\delta)$.
The reverse inequality follows from the supposition
$(\beta,\gamma,\delta)\in\Omega^{+3}_{p,\tilde{r}}$ by
\eqref{2:eq2.4}. Thus \eqref{2:eq2.6} holds for all
$s_0\in[1,\infty]$.
\end{proof}

\begin{remark}\label{2:r2.6}
Limit calculations in  \eqref{2:eq2.9} -- \eqref{2:eq2.8} are based
on the following simple fact. If a sequence of vectors
$(x_1^n,x_2^n,x_3^n)$ tends to the vector $(x_1,x_2,x_3)$, then the
sequence of their nonincreasing rearrangements
$(x_{[1]}^n,x_{[2]}^n,x_{[3]}^n)$ tends to the rearrangement
$(x_{[1]},x_{[2]},x_{[3]})$. Indeed, the classical
Hardy--Littlewood--Polya inequality
$$
\underset{i=1}{\overset{m}{\sum}}a_ib_i\leqslant
\underset{i=1}{\overset{m}{\sum}}a_{[i]}b_{[i]},\quad
a_i,b_i\in\mathbb{R},\,1\leqslant i\leqslant m,
$$
see, for example, \cite[Chapter 6, A. 3]{MO}, has as a consequence
the estimation
$$
\underset{i=1}{\overset{3}{\sum}}(x_{[i]}-x_{[i]}^n)^2\leqslant
\underset{i=1}{\overset{3}{\sum}}(x_i-x_i^n)^2.
$$
\end{remark}

\begin{corollary}\label{2:c2.7}
If the metric space $(X,d)$ is ultrametric, then all pretangent
spaces $\Omega^X_{p,\tilde{r}}$ are ultrametric for each $p\in X$.
\end{corollary}

\begin{remark}\label{2:r2.8}
The sums
$$
S_t(x)=\left(\underset{i=1}{\overset{n}{\sum}}x_i^t\right)^\frac{1}{t},
\  \  x_i\in(0,\infty),
$$
decrease from $+\infty$ to $x_1\vee\ldots\vee x_n$ when $t$
increases from 0 to $+\infty$. The inequality
$S_{t_2}(x)<S_{t_1}(x)$, $0<t_1<t_2<\infty$, is sometimes referred
to as the Jensen inequality. For the proof see, for example,
\cite{Be}.
\end{remark}

As it was shown in Proposition \ref{2:as2.5} the condition
$$
\underset{\underset{(x,y,z)\in X^{+3}}{x,y,z\rightarrow
p}}{\lim}s(x,z,y)=\infty
$$
is sufficient for the ultrametricity of all pretangent spaces
$\Omega^X_{p,\tilde{r}}$ but it is not necessary as the following
proposition shows.

\begin{proposition}\label{2:as2.9}
For every $s_0\in[1,\infty)$ there exists a metric space $(X,d)$
with a marked point $p$ such that
$$
\underset{\underset{(x,y,z)\in X^{+3}}{x,y,z\rightarrow
p}}{\lim}s(x,y,z)=s_0
$$
but all pretangent spaces $\Omega^X_{p,\tilde{r}}$ are ultrametric.
\end{proposition}

\begin{proof}
Let $\tilde{b}=\{b_n\}_{n\in\mathbb{N}}$ be a sequence of positive
real numbers such that
$\underset{n\rightarrow\infty}{\lim}\frac{b_n}{b_{n+1}}=\infty$.
Let us consider the metric space $(X,d)$ with
$$
X=\{0\}\cup\left(\underset{n\in\mathbb{N}}{\bigcup}\{b_n\}\right)
\text{ and } d(x,y)=|x-y|
$$
and with a marked point $p=0$. It simply follows from \cite{DM2}
that
$$
\card\Omega^X_{p,\tilde{b}}\leqslant 2
$$
for each pretangent space $\Omega^X_{p,\tilde{b}}$. Consequently
all these pretangent spaces are ultrametric.

Note now, that every triple $(x,y,z)\in X^3$ can be rearranged
such that $d(x,z)=d(x,y)+d(z,y)$. Consequently, for this $(X,d)$,
we have
$$
\underset{\underset{(x,y,z)\in X^{+3}}{x,y,z\rightarrow
p}}{\lim}s(x,y,z)=1,
$$
so the proposition follows for $s_0=1$. If $s_0>1$, then
$d^{\frac{1}{s_0}}$ is also a metric on $X$. The space
$(X,d^{\frac{1}{s_0}})$, the snowflaked version of $(X,d)$, is the
desirable example. This proves the proposition.
\end{proof}

Let $(X,d)$ be a metric space with a marked point $p$. Define a
function $F:X\times X\rightarrow\mathbb{R}$ by the rule
\begin{equation}\label{2:eq2.13}
F(x,y):=\begin{cases} \frac{d(x,y)(d(x,p)\wedge
d(y,p))}{(d(x,p)\vee d(y,p))^2} \qquad\text{if}\qquad
(x,y)\neq(p,p)\\
0\qquad\qquad\qquad\qquad\text{ if}\qquad (x,y)=(p,p).
\end{cases}
\end{equation}

Note that $0\leqslant F(x,y)\leqslant 2$ for all $x$ and $y$.
Write
\begin{equation}\label{2:eq2.14}
\Phi(x,y,z):=F(x,y)\vee F(x,z)\vee F(y,z)
\end{equation}
and
\begin{equation}\label{2:eq2.15}
\Psi(x,y,z):=\frac{d(x,y)\vee d(y,z)\vee d(x,z)}{d(x,y)\wedge
d(y,z)\wedge d(z,x)}
\end{equation}
for all $x,y,z\in X$ where $\Psi(x,y,z):=\infty$ if $d(x,y)\wedge
d(y,z)\wedge d(z,x)=0$.

The following theorem is an ultrametricity criterion for
pretangent spaces of general metric spaces.

\begin{theorem}\label{2:T2.10}
Let $(X,d)$ be a metric space with a marked point $p$. The
following two statements are equivalent.
\newline
(i) All pretangent spaces $\Omega^X_{p,\tilde{r}}$ are
ultrametric.
\newline
(ii) We have the limit relation
\begin{equation}\label{2:eq2.16}
\underset{x,y,z\rightarrow
p}{\lim}\frac{s(x,y,z)}{\Phi(x,y,z)}\Psi(x,y,z)=\infty
\end{equation}
where $\frac{1}{\Phi(x,y,z)}:=\infty$ if $\Phi(x,y,z)=0$.
\end{theorem}

\begin{proof}
$(i)\Rightarrow(ii)$ Suppose statement $(i)$ is true. If
\eqref{2:eq2.16} does not hold, then there are $\alpha\in(0,\infty)$
and sequences $\tilde{x}=\{x_n\}_{n\in\mathbb{N}}$,
$\tilde{y}=\{y_n\}_{n\in\mathbb{N}}$,
$\tilde{z}=\{z_n\}_{n\in\mathbb{N}}$ from $\tilde{X}$ such that
$$
\underset{n\rightarrow\infty}{\lim}x_n=
\underset{n\rightarrow\infty}{\lim}y_n=
\underset{n\rightarrow\infty}{\lim}z_n=p
$$
and that
$$
\underset{n\rightarrow\infty}{\lim}
\frac{s(x_n,y_n,z_n)}{\Phi(x_n,y_n,z_n)}\Psi(x_n,y_n,z_n)=\alpha.
$$
Since the double inequalities
$\frac{1}{2}\leqslant\frac{1}{\Phi(x,y,z)}\leqslant\infty$,
$1\leqslant\Psi(x,y,z)\leqslant\infty$ and $1\leqslant
s(x,y,z)\leqslant\infty$ hold for all $x,y,z\in X$, we can
suppose, proceeding to a subsequence if it is necessary, that
there exist the following limits
\begin{equation}\label{2:eq2.17}
\begin{array}{l}
\underset{n\rightarrow\infty}{\lim}\Psi(x_n,y_n,z_n)=:\psi_0,\qquad
\underset{n\rightarrow\infty}{\lim}\frac{1}{\Phi(x_n,y_n,z_n)}=:\phi_0,\\
\qquad\qquad\qquad\underset{n\rightarrow\infty}{\lim}s(x_n,y_n,z_n)=:s_0
\end{array}
\end{equation}
with $\infty>\psi_0\geqslant 1$, $\infty>\phi_0\geqslant\frac{1}{2}$
and with $\infty>s_0\geqslant 1$. It follows from \eqref{2:eq2.14}
that for every $n\in\mathbb{N}$ we have at least one of the
following equalities
$$
\begin{array}{l}
F(x_n,y_n)=\Phi(x_n,y_n,z_n),\ \ F(y_n,z_n)=\Phi(x_n,y_n,z_n),\\
\qquad\qquad\qquad F(z_n,x_n)=\Phi(x_n,y_n,z_n).
\end{array}
$$
Suppose the first equality
\begin{equation}\label{2:eq2.18}
F(x_n,y_n)=\Phi(x_n,y_n,z_n)
\end{equation}
holds on an infinite subset of $\mathbb{N}$. Then, passing once
again to a subsequence, we take that \eqref{2:eq2.18} is true for
every $n\in\mathbb{N}$. Hence the second equality in
\eqref{2:eq2.17} can be rewritten as
\begin{equation}\label{2:eq2.19}
\underset{n\rightarrow\infty}{\lim}F(x_n,y_n)=
\underset{n\rightarrow\infty}{\lim}\frac{d(x_n,y_n)(d(x_n,p)\wedge
d(y_n,p))}{(d(x_n,p)\vee d(y_n,p))^2}=\frac{1}{\phi_0}\in(0,2].
\end{equation}
Analogously, we can suppose that the equality
\begin{equation}\label{2:eq2.20}
d(x_n,p)=d(x_n,p)\vee d(y_n,p)
\end{equation}
holds on some infinite subset of $\mathbb{N}$ and passing to a
subsequence, that this subset equals $\mathbb{N}$. Relations
\eqref{2:eq2.19}--\eqref{2:eq2.20} imply the inequality
$d(x_n,p)>0$ for sufficiently large $n$. Write
\begin{equation}\label{2:eq2.21}
r_n:=\begin{cases}
1\qquad\qquad\text{if}\quad d(x_n,p)=0\\
d(x_n,p)\quad\text{if}\quad d(x_n,p)>0
\end{cases}
\end{equation}
and $\tilde{r}:=\{r_n\}_{n\in\mathbb{N}}$. We can now easily show
that the quantities
\begin{equation}\label{2:eq2.22}
\frac{d(x_n,p)}{r_n},\quad\frac{d(y_n,p)}{r_n},\quad
\frac{d(z_n,p)}{r_n},\quad\frac{d(x_n,y_n)}{r_n},\quad
\frac{d(x_n,z_n)}{r_n},\quad\frac{d(y_n,z_n)}{r_n}
\end{equation}
are bounded above by a constant. Indeed, for
$\frac{d(x_n,p)}{r_n}$ and $\frac{d(y_n,p)}{r_n}$ it follows from
\eqref{2:eq2.20}--\eqref{2:eq2.21} and for
$\frac{d(x_n,y_n)}{r_n}$ from the triangle inequality
$$
\frac{d(x_n,y_n)}{r_n}\leqslant\frac{d(x_n,p)}{r_n}+\frac{d(y_n,p)}{r_n}.
$$
Since $\frac{d(x_n,y_n)}{r_n}$ is bounded above and the first
limit in \eqref{2:eq2.17} is finite, the quantities
$\frac{d(y_n,z_n)}{r_n}$ and $\frac{d(x_n,z_n)}{r_n}$ are also
bounded above. Finally, the inequality
$$
d(p,z_n)\leqslant d(p,x_n)+d(x_n,z_n)
$$
implies the desirable boundedness of $\frac{d(z_n,p)}{r_n}$.

Since all quantities in \eqref{2:eq2.22} are bounded, there is a
sequence $\{n_k\}_{k\in\mathbb{N}}$ of natural numbers for which
all limits
\begin{equation}\label{2:eq2.23}
\begin{array}{l}
\underset{k\rightarrow\infty}{\lim}\frac{d(x_{n_k},p)}{r_{n_k}},\quad
\underset{k\rightarrow\infty}{\lim}\frac{d(y_{n_k},p)}{r_{n_k}},\quad
\underset{k\rightarrow\infty}{\lim}\frac{d(z_{n_k},p)}{r_{n_k}},\\
\underset{k\rightarrow\infty}{\lim}\frac{d(x_{n_k},y_{n_k})}{r_{n_k}},\quad
\underset{k\rightarrow\infty}{\lim}\frac{d(x_{n_k},z_{n_k})}{r_{n_k}},\quad
\underset{k\rightarrow\infty}{\lim}\frac{d(y_{n_k},z_{n_k})}{r_{n_k}}
\end{array}
\end{equation}
are finite. Renaming $\tilde{x}:=\{x_{n_k}\}_{k\in\mathbb{N}}$,
$\tilde{y}:=\{y_{n_k}\}_{k\in\mathbb{N}}$,
$\tilde{z}:=\{z_{n_k}\}_{k\in\mathbb{N}}$ and
$\tilde{r}:=\{r_{n_k}\}_{k\in\mathbb{N}}$ we obtain that
$\tilde{x},\tilde{y},\tilde{z}$ and $\tilde{p}$ are mutually
stable w.r.t. $\tilde{r}$. We can now easily show that
\begin{equation}\label{2:eq2.24}
\tilde{d}_{\tilde{r}}(\tilde{x},\tilde{y})\neq 0
\end{equation}
and that
\begin{equation}\label{2:eq2.25}
\tilde{d}_{\tilde{r}}(\tilde{x},\tilde{z})\neq
0\neq\tilde{d}_{\tilde{r}}(\tilde{y},\tilde{z}).
\end{equation}
For this purpose, note that \eqref{2:eq2.19} -- \eqref{2:eq2.21}
imply
$$
\begin{array}{l}
\underset{n\rightarrow\infty}{\lim}F(x_n,y_n)=\underset{n\rightarrow\infty}{\lim}
\frac{d(x_n,y_n)}{r_n}\left(\frac{d(x_n,p)}{r_n}\wedge
\frac{d(p,y_n)}{r_n}\right)\\\qquad\qquad\qquad=\underset{n\rightarrow\infty}{\lim}
\frac{d(x_n,y_n)}{r_n}\frac{d(p,y_n)}{r_n}=\tilde{d}_{\tilde{r}}(\tilde{x},\tilde{y})
\tilde{d}_{\tilde{r}}(\tilde{p},\tilde{y})=\frac{1}{\phi_0}\in(0,2],
\end{array}
$$
consequently relation \eqref{2:eq2.24} holds. Moreover
\eqref{2:eq2.24} and the finiteness of the first limit in
\eqref{2:eq2.17} imply \eqref{2:eq2.25}.

Let $(\Omega_{p,\tilde{r}}^X,\rho)$ be a pretangent space such
that $\tilde{x}\in\beta$, $\tilde{y}\in\gamma$ и
$\tilde{z}\in\delta$ for some
$\beta,\gamma,\delta\in\Omega_{p,\tilde{r}}$. The definition of
the function $(x,y,z)\mapsto s(x,y,z)$ and the finiteness of the
last limit in \eqref{2:eq2.17} imply the equality
$$
\rho(\beta,\gamma)=\rho(\beta,\gamma)\vee \rho(\gamma,\delta)\vee
\rho(\delta,\beta)
$$
and, in addition, it follows from \eqref{2:eq2.24} --
\eqref{2:eq2.25} that $\rho(\beta,\gamma)\wedge
\rho(\gamma,\delta)\wedge \rho(\delta,\beta)>0$. We may assume,
without loss of generality, that
$$
\rho(\beta,\gamma)\geqslant\rho(\gamma,\delta)\geqslant\rho(\delta,\beta)>0,
$$
that is $(\gamma,\beta,\delta)\in\Omega_{p,\tilde{r}}^{+3}$. Using
the last limit relation in \eqref{2:eq2.17} and reasoning as in
the proof of Proposition \ref{2:as2.5} we obtain the equality
\begin{equation}\label{2:eq2.26}
\rho(\beta,\gamma)=\left((\rho(\beta,\delta))^{s_0}+(\rho(\delta,\gamma))^{s_0}
\right)^\frac{1}{s_0}.
\end{equation}
Since $s_0\in[1,\infty)$, the last equality shows that
$(\Omega_{p,\tilde{r}}^X,\rho)$ is not an ultrametric space
contrary to the assumption. To complete the proof of the
implication $(i)\Rightarrow(ii)$ it suffices to observe that
\eqref{2:eq2.26} was derived from relations \eqref{2:eq2.24} and
\eqref{2:eq2.25} and that these two relations remain valid if the
pair $(x_n,y_n)$ in \eqref{2:eq2.18} is replaced by an arbitrary
pair from the set
$\{(x_n,z_n),(y_n,x_n),(y_n,z_n),(z_n,x_n),(z_n,y_n)\}$.

$(ii)\Rightarrow(i)$ Suppose now that \eqref{2:eq2.16} holds. We
must prove that all pretangent spaces
$(\Omega_{p,\tilde{r}}^X,\rho)$ are ultrametric. To this end it
suffices to show that for an arbitrary normalizing sequence
$\tilde{r}$ the inequality
\begin{equation}\label{2:eq2.27}
\tilde{d}_{\tilde{r}}(\tilde{x},\tilde{y})\leqslant
\tilde{d}_{\tilde{r}}(\tilde{x},\tilde{z})\vee
\tilde{d}_{\tilde{r}}(\tilde{z},\tilde{y})
\end{equation}
holds for all mutually stable (w.r.t. $\tilde{r}$)
$\tilde{x},\tilde{y},\tilde{z}\in\tilde{X}$ whenever
\begin{equation}\label{2:eq2.28}
\tilde{d}_{\tilde{r}}(\tilde{x},\tilde{y})\geqslant
\tilde{d}_{\tilde{r}}(\tilde{x},\tilde{z})\geqslant
\tilde{d}_{\tilde{r}}(\tilde{z},\tilde{y})>0
\end{equation}
and whenever there are finite limits
\begin{equation}\label{2:eq2.29}
\tilde{d}_{\tilde{r}}(\tilde{x},\tilde{p})=\underset{n\rightarrow\infty}{\lim}\frac{d(x_n,p)}{r_n},\quad
\tilde{d}_{\tilde{r}}(\tilde{y},\tilde{p})=\underset{n\rightarrow\infty}{\lim}\frac{d(y_n,p)}{r_n},\quad
\tilde{d}_{\tilde{r}}(\tilde{z},\tilde{p})=\underset{n\rightarrow\infty}{\lim}\frac{d(z_n,p)}{r_n}
\end{equation}
where $\{x_n\}_{n\in\mathbb{N}}=\tilde{x}$,
$\{y_n\}_{n\in\mathbb{N}}=\tilde{y}$,
$\{z_n\}_{n\in\mathbb{N}}=\tilde{z}$. Limit relation
\eqref{2:eq2.16}, the definition of $\Psi$ and inequalities
\eqref{2:eq2.28} imply
$$
\infty=\underset{n\rightarrow\infty}{\lim}\frac{s(x_n,y_n,z_n)}{\Phi(x_n,y_n,z_n)}\Psi(x_n,y_n,z_n)=
\frac{\tilde{d}_{\tilde{r}}(\tilde{x},\tilde{y})}{\tilde{d}_{\tilde{r}}(\tilde{y},\tilde{z})}
\underset{n\rightarrow\infty}{\lim}\frac{s(x_n,y_n,z_n)}{\Phi(x_n,y_n,z_n)}.
$$
Consequently we have
\begin{equation}\label{2:eq2.30}
\underset{n\rightarrow\infty}{\lim}\frac{s(x_n,y_n,z_n)}{\Phi(x_n,y_n,z_n)}=\infty
\end{equation}
because, by \eqref{2:eq2.28}, the quantity
$\frac{\tilde{d}_{\tilde{r}}(\tilde{x},\tilde{y})}{\tilde{d}_{\tilde{r}}(\tilde{y},\tilde{z})}$
are finite and positive. If, in addition, the equality
\begin{equation}\label{2:eq2.31}
\underset{n\rightarrow\infty}{\lim}s(x_n,y_n,z_n)=\infty
\end{equation}
holds, then reasoning as in the proof of the second part of
Proposition \ref{2:as2.5} we obtain inequality \eqref{2:eq2.27}.
If \eqref{2:eq2.31} does not hold, then there are an infinite
strictly increasing sequence $\{n_k\}_{k\in\mathbb{N}}$ of natural
numbers and a constant $s_0\in[1,\infty)$ such that
$$
\underset{k\rightarrow\infty}{\lim}s(x_{n_k},y_{n_k},z_{n_k})=s_0.
$$
The last equality and \eqref{2:eq2.30} have as a consequence
\begin{equation}\label{2:eq2.32}
\underset{k\rightarrow\infty}{\lim}\frac{1}{\Phi(x_{n_k},y_{n_k},z_{n_k})}=\infty.
\end{equation}
It follows from this and \eqref{2:eq2.13}--\eqref{2:eq2.14} that
$$
0=\underset{k\rightarrow\infty}{\lim} \frac{(d(x_{n_k},p)\wedge
d(y_{n_k},p))d(x_{n_k},y_{n_k})}{(d(x_{n_k},p)\vee
d(p,y_{n_k}))^2}=\tilde{d}_{\tilde{r}}(\tilde{x},\tilde{y})
\underset{k\rightarrow\infty}{\lim} \frac{
\frac{d(x_{n_k},p)}{r_{n_k}}\wedge\frac{d(y_{n_k},p)}{r_{n_k}}}
{\left(\frac{d(x_{n_k},p)}{r_{n_k}}\vee\frac{d(y_{n_k},p)}{r_{n_k}}\right)^2}
$$
and consequently
$$
\underset{k\rightarrow\infty}{\lim} \frac{
\frac{d(x_{n_k},p)}{r_{n_k}}\wedge\frac{d(y_{n_k},p)}{r_{n_k}}}
{\left(\frac{d(x_{n_k},p)}{r_{n_k}}\vee\frac{d(y_{n_k},p)}{r_{n_k}}\right)^2}=0.
$$
Similarly we have
\begin{equation}\label{2:eq2.33}
\underset{k\rightarrow\infty}{\lim} \frac{
\frac{d(x_{n_k},p)}{r_{n_k}}\wedge\frac{d(z_{n_k},p)}{r_{n_k}}}
{\left(\frac{d(x_{n_k},p)}{r_{n_k}}\vee\frac{d(z_{n_k},p)}{r_{n_k}}\right)^2}=
\underset{k\rightarrow\infty}{\lim} \frac{
\frac{d(y_{n_k},p)}{r_{n_k}}\wedge\frac{d(z_{n_k},p)}{r_{n_k}}}
{\left(\frac{d(y_{n_k},p)}{r_{n_k}}\vee\frac{d(z_{n_k},p)}{r_{n_k}}\right)^2}=0.
\end{equation}
Note that
\begin{equation}\label{2:eq2.34}
\tilde{d}_{\tilde{r}}(\tilde{x},\tilde{p})\vee
\tilde{d}_{\tilde{r}}(\tilde{y},\tilde{p})\vee
\tilde{d}_{\tilde{r}}(\tilde{z},\tilde{p})>0
\end{equation}
because in the opposite case the triangle inequality implies
$\tilde{d}_{\tilde{r}}(\tilde{x},\tilde{y})=\tilde{d}_{\tilde{r}}(\tilde{x},\tilde{z})=
\tilde{d}_{\tilde{r}}(\tilde{z},\tilde{y})=0$, contrary to
\eqref{2:eq2.28}. Suppose that
\begin{equation}\label{2:eq2.35}
\tilde{d}_{\tilde{r}}(\tilde{x},\tilde{p})\vee
\tilde{d}_{\tilde{r}}(\tilde{y},\tilde{p})\vee
\tilde{d}_{\tilde{r}}(\tilde{z},\tilde{p})=\tilde{d}_{\tilde{r}}(\tilde{z},\tilde{p}).
\end{equation}
This equality and \eqref{2:eq2.33} imply
$$
0=\underset{k\rightarrow\infty}{\lim}\frac{1}{(\tilde{d}_{\tilde{r}}(\tilde{z},\tilde{p}))^2}
\left(\frac{d(x_{n_k},p)}{r_{n_k}}\wedge\frac{d(z_{n_k},p)}{r_{n_k}}\right)=
\frac{\tilde{d}_{\tilde{r}}(\tilde{x},\tilde{p})}{(\tilde{d}_{\tilde{r}}(\tilde{z},\tilde{p}))^2},
$$
i.e. $\tilde{d}_{\tilde{r}}(\tilde{x},\tilde{p})=0$. Completely
analogously we have
$\tilde{d}_{\tilde{r}}(\tilde{y},\tilde{p})=0$. It means that
$\tilde{d}_{\tilde{r}}(\tilde{x},\tilde{y})=0$ which implies
\eqref{2:eq2.27}. It still remains to note that similar arguments
are applicable if the right side of \eqref{2:eq2.35} is
$\tilde{d}_{\tilde{r}}(\tilde{x},\tilde{p})$ or
$\tilde{d}_{\tilde{r}}(\tilde{y},\tilde{p})$ instead of
$\tilde{d}_{\tilde{r}}(\tilde{z},\tilde{p})$. Hence in all cases
\eqref{2:eq2.16} implies \eqref{2:eq2.27}.
\end{proof}


\section{Metric betweenness in pretangent spaces}
The purpose of this part of the paper is to obtain an analog of
Theorem \ref{2:T2.10} for the pretangent spaces which are not
ultrametric. Recall the following definition, see, for example,
\cite[p. 55]{Pa}.

\begin{definition}\label{3:d3.1}
Let $(X,d)$ be a metric space and let $x,y,z$ be distinct points
of $X$. The point $y$ lies between points $x$ and $z$ if
\begin{equation}\label{3:eq3.1}
d(x,z)=d(x,y)+d(y,z).
\end{equation}
\end{definition}

Denote by $\mathfrak{M}$ the class of all metric spaces $(X,d)$
such that \eqref{3:eq3.1} holds for all $x,y,z\in X$ whenever
$d(x,z)\geqslant d(x,y)\geqslant d(y,z)$.

It is easy to see that $X\in\mathfrak{M}$ if and only if the
betweenness exponent $t_0(Y)=1$ for each $Y\subseteq X$ with $\card
Y\geqslant 3$. Proposition \ref{2:as2.5} shows that
$t_0(\Theta)=s_0$ for every subspace $\Theta$ of pretangent space
$(\Omega_{p,\tilde{r}},\rho)$ provided that  $\card\Theta\geqslant
3$ and limit relation \eqref{2:eq2.5} holds with $s_0<\infty$. Thus
the spaces $(\Omega_{p,\tilde{r}}^X,\rho^{s_0})$ belong to
$\mathfrak{M}$ under these conditions.

\begin{theorem}\label{3:T3.2}
Let $(X,d)$ be a metric space, $p$ a limit point of $X$ and $s_1$
a positive number. The membership relation
\begin{equation}\label{3:eq3.2}
(\Omega_{p,\tilde{r}}^X,\rho^{s_1})\in\mathfrak{M}
\end{equation}
holds for every pretangent space $\Omega_{p,\tilde{r}}^X$ if and
only if
\begin{equation}\label{3:eq3.3}
\underset{\underset{(x,y,z)\in X^{+3}}{x,y,z\rightarrow p}}{\lim}
\frac{\Psi(x,y,z)s^2(x,y,z)}{\Phi(x,y,z)(s_1-s(x,y,z))^2}=\infty
\end{equation}
where $\left(\frac{s(x,y,z)}{s_1-s(x,y,z)}\right)^2:=1$ in the
case $s(x,y,z)=\infty$.
\end{theorem}

\begin{remark}\label{3:r3.3}
Membership relation \eqref{3:eq3.2} means, in particular, that
$\rho^{s_1}$ is a metric on $\Omega_{p,\tilde{r}}^X$. If
$\card\Omega_{p,\tilde{r}}^X\geqslant 3$, then using Lemma
\ref{2:L2.2} we see that $s_1$ equals the betweenness exponent of
$(\Omega_{p,\tilde{r}}^X,\rho)$ whenever \eqref{3:eq3.2} holds.
Moreover \eqref{3:eq3.2} holds for all $s_1>0$ if and only if
$\card\Omega_{p,\tilde{r}}^X\leqslant 2$.
\end{remark}

The following proof succeeds the scheme of the proof of Theorem
\ref{2:T2.10}

\begin{proof}[Proof of Theorem \ref{3:T3.2}.]
Suppose that \eqref{3:eq3.2} holds for all pretangent spaces
$\Omega_{p,\tilde{r}}^X$. We must prove that \eqref{3:eq3.3}
holds. The direct calculations show that
\begin{equation}\label{3:eq3.4}
\frac{(s(x,y,z))^2}{(s_1-s(x,y,z))^2}\geqslant
\begin{cases}
1\qquad\qquad\text{ if } s_1=1\\
1\wedge\frac{1}{(1-s_1)^2}\text{ if }
s_1\in(0,\infty)\diagdown\{1\}.
\end{cases}
\end{equation}
Hence if \eqref{3:eq3.3} does not hold, then as in the case of
\eqref{2:eq2.17} there is a sequence of triples $(x_n,y_n,z_n)\in
X^{+3}$, $n\in\mathbb{N}$, such that
$$
\underset{n\rightarrow\infty}{\lim}x_n=\underset{n\rightarrow\infty}{\lim}y_n=
\underset{n\rightarrow\infty}{\lim}z_n=p
$$
and that the following finite positive limits exist
\begin{equation}\label{3:eq3.5}
\begin{array}{l}
\underset{n\rightarrow\infty}{\lim}\Psi(x_n,y_n,z_n)=\psi_0,\qquad
\underset{n\rightarrow\infty}{\lim}\frac{1}{\Phi(x_n,y_n,z_n)}=\phi_0,\\
\qquad\qquad\qquad\underset{n\rightarrow\infty}{\lim}\frac{s^2(x_n,y_n,z_n)}{(s_1-s(x_n,y_n,z_n))^2}=s_2.
\end{array}
\end{equation}

The condition $(x_n,y_n,z_n)\in X^{+3}$ and \eqref{2:eq2.15} imply
the equality
$$
\Psi(x_n,y_n,z_n)=\frac{d(x_n,z_n)}{d(y_n,z_n)}
$$
and the membership relations
$$
\frac{d(x_n,z_n)}{d(x_n,y_n)},\frac{d(x_n,y_n)}{d(y_n,z_n)}\in\left[1,\frac{d(x_n,z_n)}{d(y_n,z_n)}\right].
$$
Consequently, replacing $\mathbb{N}$ by a suitable subset, we may
suppose that there are finite constant $\psi_1,\psi_2$ such that
\begin{equation}\label{3:eq3.6}
\underset{n\rightarrow\infty}{\lim}\frac{d(x_n,z_n)}{d(x_n,y_n)}:=\psi_1,\quad
\underset{n\rightarrow\infty}{\lim}\frac{d(x_n,y_n)}{d(y_n,z_n)}:=\psi_2
\end{equation}
and $\psi_1\psi_2=\psi_0$ and $\psi_1\wedge\psi_2\geqslant 1$.

Using \eqref{3:eq3.5} and \eqref{3:eq3.6} and repeating the
arguments from the first part of the proof of Theorem
\ref{2:T2.10} we find a normalizing sequence
$\tilde{r}=\{r_n\}_{n\in\mathbb{N}}$ such that there are finite
limits
\begin{equation}\label{3:eq3.7}
\begin{array}{l}
\tilde{d}(\tilde{p},\tilde{x})=\underset{n\rightarrow\infty}{\lim}\frac{d(x_n,p)}{r_n},\quad
\tilde{d}(\tilde{p},\tilde{y})=\underset{n\rightarrow\infty}{\lim}\frac{d(y_n,p)}{r_n},\quad
\tilde{d}(\tilde{p},\tilde{z})=\underset{n\rightarrow\infty}{\lim}\frac{d(z_n,p)}{r_n},\\
\tilde{d}(\tilde{x},\tilde{y})=\underset{n\rightarrow\infty}{\lim}\frac{d(x_n,y_n)}{r_n},\quad
\tilde{d}(\tilde{x},\tilde{z})=\underset{n\rightarrow\infty}{\lim}\frac{d(x_n,z_n)}{r_n},\quad
\tilde{d}(\tilde{y},\tilde{z})=\underset{n\rightarrow\infty}{\lim}\frac{d(y_n,z_n)}{r_n}
\end{array}
\end{equation}
and, in addition, the inequalities
\begin{equation}\label{3:eq3.8}
\tilde{d}(\tilde{x},\tilde{z})\geqslant\tilde{d}(\tilde{x},\tilde{y})
\geqslant\tilde{d}(\tilde{y},\tilde{z})>0
\end{equation}
hold. Note that \eqref{3:eq3.7} is similar to \eqref{2:eq2.23} and
\eqref{3:eq3.8} can be obtained as \eqref{2:eq2.24} ---
\eqref{2:eq2.25}.

Let us consider now the limit relation
\begin{equation}\label{3:eq3.9}
\underset{n\rightarrow\infty}{\lim}\frac{s^2(x_n,y_n,z_n)}{(s_1-s(x_n,y_n,z_n))^2}=s_2.
\end{equation}
We first establish that $s_2\neq 1$. As usual, replacing
$\mathbb{N}$ by a suitable infinite subset, we may suppose that
there is the limit
\begin{equation}\label{3:eq3.10}
\underset{n\rightarrow\infty}{\lim}s(x_n,y_n,z_n)=s_0\in[1,\infty].
\end{equation}
If $s_0=+\infty$, then, as in the proof of Proposition
\ref{2:as2.5}, we can obtain the equality
\begin{equation}\label{3:eq3.10*}
\tilde{d}(\tilde{x},\tilde{z})=\tilde{d}(\tilde{x},\tilde{y})\vee
\tilde{d}(\tilde{y},\tilde{z}).
\end{equation}
Furthermore, it follows from \eqref{3:eq3.2} that
\begin{equation}\label{3:eq3.11}
\left(\tilde{d}(\tilde{x},\tilde{z})\right)^{s_1}=
\left(\tilde{d}(\tilde{x},\tilde{y})\right)^{s_1}+
\left(\tilde{d}(\tilde{y},\tilde{z})\right)^{s_1}.
\end{equation}
Equalities \eqref{3:eq3.10*} and \eqref{3:eq3.11} imply the equality
$\tilde{d}(\tilde{y},\tilde{z})=0$, contrary to \eqref{3:eq3.8}.
Hence the limit in \eqref{3:eq3.10} is finite. If $s_2=1$, then
using \eqref{3:eq3.9} we have
$$
\frac{s^2_0}{(s_0-s_1)^2}=1,
$$
that is $s_0=s_1/2$. From this we obtain
\begin{equation}\label{3:eq3.12}
\left(\tilde{d}(\tilde{x},\tilde{z})\right)^{\frac{s_1}{2}}=
\left(\tilde{d}(\tilde{x},\tilde{y})\right)^{\frac{s_1}{2}}+
\left(\tilde{d}(\tilde{y},\tilde{z})\right)^{\frac{s_1}{2}}
\end{equation}
in the same manner as in the case $s_0=\infty$. It contradicts
\eqref{3:eq3.11} because the function
\begin{equation}\label{3:eq3.13}
f(s)= \left((\tilde{d}(\tilde{x},\tilde{y})^s+
(\tilde{d}(\tilde{y},\tilde{z}))^s\right)^{\frac{1}{s}},
\end{equation}
is strictly decreasing in $s$. Thus $s_2\neq 1$ as it was indicated.

Let us analyze now the possible value of $s_0$ in \eqref{3:eq3.10}.
It was noted above that $s_0\neq\infty$ because $s_2\neq 1$. Hence
$s_0$ is a root of the equation
$$
\frac{x^2}{(x-s_1)^2}=s_2
$$
where $s_2\in(1,\infty)$. Since $s_2\neq 1$, two possible values
of $s_0$ are
\begin{equation}\label{3:eq3.14}
\frac{s_1\sqrt{s_2}}{1+\sqrt{s_2}}\quad\text{ and
}\quad\frac{-s_1\sqrt{s_2}}{1-\sqrt{s_2}},
\end{equation}
where we put $\sqrt{s_2}>0$. It follows from \eqref{3:eq3.14} that
$s_0\neq s_1$ because if $s_0=s_1$, then
$$
s_1(1+\sqrt{s_2})=s_1\sqrt{s_2}\quad\text{or}\quad
s_1(1-\sqrt{s_2})=-s_1\sqrt{s_2},
$$
i.e., $s_1=0$ which contradicts the conditions of the theorem. As
in \eqref{3:eq3.12} we obtain
$$
\left(\tilde{d}(\tilde{x},\tilde{z})\right)^{s_0}=
\left(\tilde{d}(\tilde{x},\tilde{y})\right)^{s_0}+
\left(\tilde{d}(\tilde{y},\tilde{z})\right)^{s_0}.
$$
The last equality contradicts \eqref{3:eq3.11} because function
\eqref{3:eq3.13} is strictly decreasing on $(0,\infty)$ and
$s_0\neq s_1$.

It follows that the positive constant $s_2$ in \eqref{3:eq3.5} and
\eqref{3:eq3.9} cannot be finite, contrary to the assumption. Thus
the limit relation \eqref{3:eq3.3} holds if all spaces
$(\Omega_{p,\tilde{r}}^X,\rho^{s_1})$ belong to $\mathfrak{M}$.

Suppose now that limit relation \eqref{3:eq3.3} holds. We must prove
that for every pretangent space $\Omega_{p,\tilde{r}}^X$ the
equality
\begin{equation}\label{3:eq3.15}
(\rho(\beta,\delta))^{s_1}=(\rho(\beta,\gamma))^{s_1}+(\rho(\gamma,\delta))^{s_1}
\end{equation}
holds for all $\beta,\gamma,\delta\in\Omega_{p,\tilde{r}}^X$
whenever
\begin{equation}\label{3:eq3.16}
\rho(\beta,\delta)\geqslant\rho(\beta,\gamma)\geqslant\rho(\gamma,\delta).
\end{equation}
Since \eqref{3:eq3.16} is trivial for $\rho(\gamma,\delta)=0$, we
assume
\begin{equation}\label{3:eq3.17}
\rho(\gamma,\delta)>0.
\end{equation}
Note that \eqref{3:eq3.16} together with \eqref{3:eq3.17} are an
equivalent of $(\beta,\gamma,\delta)\in\Omega_{p,\tilde{r}}^{+3}$.

Let $(\beta,\gamma,\delta)\in\Omega_{p,\tilde{r}}^{+3}$ is given
and let
\begin{equation}\label{3:eq3.18}
\tilde{x}^*=\{x^*_n\}_{n\in\mathbb{N}}\in\beta,\quad
\tilde{y}^*=\{y^*_n\}_{n\in\mathbb{N}}\in\gamma,\quad
\tilde{z}^*=\{z^*_n\}_{n\in\mathbb{N}}\in\delta.
\end{equation}
At least one of the six rearrangements of the points
$x^*_n,y^*_n,z^*_n$ belongs to $X^{+3}$ for $n$ belonging to some
infinite subsequence $\{n_k\}_{k\in\mathbb{N}}$ of natural numbers.
We denote by $x_n$ the first element of this rearrangement, by $y_n$
the second and by $z_n$ the third one. Using, as usual,
$\{n_k\}_{k\in\mathbb{N}}$ instead of  the sequence of all natural
numbers we may suppose that $(x_k,y_k,z_k)\in X^{+3}$ for each
$k\in\mathbb{N}$. Write $\tilde{x}:=\{x_k\}_{k\in\mathbb{N}}$,
$\tilde{y}:=\{y_k\}_{k\in\mathbb{N}}$,
$\tilde{z}:=\{z_k\}_{k\in\mathbb{N}}$. Relations \eqref{3:eq3.18}
imply the equalities
\begin{equation}\label{3:eq3.19}
\begin{array}{l}
\tilde{d}_{\tilde{r}'}(\tilde{x},\tilde{z})=
\underset{k\rightarrow\infty}{\lim}\frac{d(x_k,z_k)}{r_{n_k}}=\rho(\beta,\delta),\quad
\tilde{d}_{\tilde{r}'}(\tilde{x},\tilde{y})=
\underset{k\rightarrow\infty}{\lim}\frac{d(x_k,y_k)}{r_{n_k}}=\rho(\beta,\gamma),\\
\qquad\qquad\qquad\qquad
\tilde{d}_{\tilde{r}'}(\tilde{y},\tilde{z})=
\underset{k\rightarrow\infty}{\lim}\frac{d(y_k,z_k)}{r_{n_k}}=\rho(\gamma,\delta),
\end{array}
\end{equation}
see Remark \ref{2:r2.6}. Moreover, it follows directly from
definitions of $\tilde{x},\tilde{y},\tilde{z}$ that there are
finite limits
$$
\underset{k\rightarrow\infty}{\lim}\frac{d(x_k,p)}{r_{n_k}},\quad
\underset{k\rightarrow\infty}{\lim}\frac{d(y_k,p)}{r_{n_k}}\quad\text{and}\quad
\underset{k\rightarrow\infty}{\lim}\frac{d(z_k,p)}{r_{n_k}}.
$$
Consequently the family
$\{\tilde{x},\tilde{y},\tilde{z},\tilde{p}'\}$ is self-stable
w.r.t. $\tilde{r}'=\{r_{n_k}\}_{k\in\mathbb{N}}$.

Using \eqref{3:eq3.19} we can rewrite \eqref{3:eq3.15} in the
equivalent form
\begin{equation}\label{3:eq3.20}
\left(\tilde{d}_{\tilde{r}'}(\tilde{x},\tilde{z})\right)^{s_1} =
\left(\tilde{d}_{\tilde{r}'}(\tilde{x},\tilde{y})\right)^{s_1} +
\left(\tilde{d}_{\tilde{r}'}(\tilde{y},\tilde{z})\right)^{s_1}.
\end{equation}
Reasoning as in the proof of \eqref{2:eq2.30} we obtain
$$
\underset{k\rightarrow\infty}{\lim}
\frac{s^2(x_k,y_k,z_k)}{\Phi(x_k,y_k,z_k)(s_1-s(x_k,y_k,z_k))^2}=\infty.
$$
If, in addition, the equality
$$
\underset{k\rightarrow\infty}{\lim}s(x_k,y_k,z_k)=s_1
$$
is true, then, using the first part of the proof of Proposition
\ref{2:as2.5}, we see that \eqref{3:eq3.20} holds. In the opposite
case there is an infinite strictly increasing sequence of natural
numbers $k_m$, $m\in\mathbb{N}$, for which, similarly
\eqref{2:eq2.32}, we obtain
$$
\underset{m\rightarrow\infty}{\lim}
\frac{1}{\Phi(x_{k_m},y_{k_m},z_{k_m})}=\infty.
$$
Now the equality \eqref{3:eq3.20} can be proved by the repetition
of the arguments which stay after \eqref{2:eq2.32}.
\end{proof}

In the following corollary and further we assume that $\mathbb{R}$
is the set of all real numbers with the usual metric
$d(x,y)=|x-y|$.

\begin{corollary}\label{3:c3.4}
Let $(X,d)$ be a metric space with a limit point $p$,
$(\Omega_{p,\tilde{r}}^X,\rho)$ a pretangent space to $X$ at the
point $p$ and $s_1$ a positive number. If \eqref{3:eq3.3} holds
for this $s_1$, then each one from the following conditions is
sufficient that the space $(\Omega_{p,\tilde{r}}^X,\rho)$ be
tangent.
\newline
(i) There is no any isometric embeddings of
$(\Omega_{p,\tilde{r}}^X,\rho^{s_1})$ in $\mathbb{R}$.
\newline
(ii) The space $(\Omega_{p,\tilde{r}}^X,\rho^{s_1})$ is isometric
to $\mathbb{R}$.
\end{corollary}

To prove this we will use the following particular case of the
classical Menger's result on the isometric embeddings into
Euclidean spaces.

\begin{theorem}[\it K. Menger]\label{3:T3.6}
Let $Y\in\mathfrak{M}$ be a metric space with $\card Y\geqslant
5$. Then $Y$ is isometric to some subset of $\mathbb{R}$.
\end{theorem}

\begin{proof}[\it Proof of Corollary \ref{3:c3.4}]
Suppose that condition $(i)$ is fulfilled and \eqref{3:eq3.3} is
true. Since \eqref{3:eq3.3} holds, we have, by Theorem \ref{3:T3.2},
that $(\Omega_{p,\tilde{r}}^X,\rho^{s_1})\in\mathfrak{M}$. Using
Menger's theorem \ref{3:T3.6} we see that
\begin{equation}\label{3:eq3.21}
\card\Omega_{p,\tilde{r}}=4.
\end{equation}
If $(\Omega_{p,\tilde{r}}^X,\rho)$ is not tangent, then there is a
pretangent space $\Omega_{p,\tilde{r}'}^X$ such that
\begin{equation}\label{3:eq3.22}
\Omega_{p,\tilde{r}'}\diagdown
em'(\Omega_{p,\tilde{r}})\neq\varnothing
\end{equation}
for the isometric embedding
$em':\Omega_{p,\tilde{r}}^X\rightarrow\Omega_{p,\tilde{r}'}^X$, see
Definition \ref{1:d1.4}. Since
$(\Omega_{p,\tilde{r}'}^X,\rho^{s_1})$ also belongs to
$\mathfrak{M}$, relations \eqref{3:eq3.21} -- \eqref{3:eq3.22} imply
the inequality
$$
\card\Omega_{p,\tilde{r}'}\geqslant 5.
$$
Hence, by Menger's theorem, there exists an isometric embedding
$f$ of $(\Omega_{p,\tilde{r}'}^X,\rho^{s_1})$ into $\mathbb{R}$.
Now the superposition
$\Omega_{p,\tilde{r}}^X\overset{em'}{\rightarrow}\Omega_{p,\tilde{r}'}^X\overset{f}{\rightarrow}\mathbb{R}$
is an isometric embedding of $(\Omega_{p,\tilde{r}}^X,\rho^{s_1})$
in $\mathbb{R}$, contrary to $(i)$.

Let condition $(ii)$ and \eqref{3:eq3.3} be fulfilled. In order
that $(\Omega_{p,\tilde{r}}^X,\rho)$ be tangent, it suffices to
show each isometric embedding $f:\mathbb{R}\rightarrow Y$ is a
bijection if $Y\in\mathfrak{M}$. Suppose an isometric embedding
$f:\mathbb{R}\rightarrow Y$, $Y\in\mathfrak{M}$, is not bijective.
Let $a\in Y\diagdown f(\mathbb{R})$ and let $b\in f(\mathbb{R})$.
Write $s:=\rho(a,b)$ where $\rho$ is the metric on $Y$. Then there
are two distinct points $p_1,p_2\in\mathbb{R}$ for which
$$
|f^{-1}(b)-p_1|=|f^{-1}(b)-p_2|=s.
$$
Consequently we have
$$
\rho(b,f(p_1))=\rho(b,f(p_2))=\rho(b,a)=s>0
$$
where all three points $f(p_1),f(p_2)$ and $b$ are distinct.
Theorem \ref{3:T3.6} implies that $(Y,\rho)$ can be isometrically
embedded in $\mathbb{R}$. Let $g:Y\rightarrow\mathbb{R}$ be a such
embedding. Then the points $g(b)$, $g(f(p_1))$ and $g(f(p_2))$ are
distinct points of the sphere $\{x\in\mathbb{R}:|g(a)-x|=s\}$. It
is impossible because every "sphere" in $\mathbb{R}$ contains only
two points.
\end{proof}

\begin{remark}\label{3:r3.7}
The four-point metric spaces $Y\in\mathfrak{M}$ which cannot be
isometrically embedded in $\mathbb{R}$ are sometimes referred to
as pseudo-linear quadruples. It is well known that $(Y,d)$ is a
pseudo-linear quadruples if and only if the points of $Y$ can be
labelled $p_0,p_1,p_2,p_3$ such that
\begin{equation}\label{3:eq3.23}
\begin{array}{l}
d(p_0,p_1)=d(p_2,p_3)=s,\qquad d(p_1,p_2)=d(p_0,p_3)=t, \\
\qquad\qquad\qquad d(p_0,p_2)=d(p_1,p_3)=s+t
\end{array}
\end{equation}
where $s$ and $t$ are some positive constants. See, for example,
\cite[p. 114]{Bl}.
\end{remark}

Each pseudo-linear quadruples can be easily realized as a subset
of the two-dimensional linear normed space $l^2_{\infty}$ where,
as usual, the norm is
$$
\|(x_1,x_2)\|=|x_1|\vee|x_2|.
$$
To this end we put
\begin{equation}\label{3:eq3.24}
p_0=(0,0), \ \ p_1=(s,s), \ \ p_2=(s+t,s-t), \ \ p_3=(t,-t)
\end{equation}
Simple calculations show that equalities \eqref{3:eq3.23} hold in
this case.

Realization \eqref{3:eq3.24} leads us to the examples of metric
spaces which have the pseudo-linear quadruples as tangent spaces.

\begin{example}\label{3:ex3.7}
Let $\tilde{r}=\{r_n\}_{n\in\mathbb{N}}$ be a decreasing sequence
of positive numbers such that
$\underset{n\rightarrow\infty}{\lim}r_n/r_{n+1}=+\infty$ and let
$t,s\in(0,\infty)$. Write $Y:=\{p_0,p_1,p_2,p_3\}$ where $p_i$,
$i=0,\ldots,3$ are the points of $l^2_{\infty}$ which were defined
by \eqref{3:eq3.24} and set
$$
X=\underset{n\in\mathbb{N}}{\bigcup}r_nY
$$
with $r_nY=\{r_np:p\in Y\}$. Consider the metric space $(X,d)$
with the metric $d$ induced from the space $l^2_{\infty}$. It is
easy to see that the sequences
$\tilde{x}_i:=\{r_np_i\}_{n\in\mathbb{N}}$, $i=0,\ldots,3$, are
mutually stable w.r.t. the normalizing sequence $\tilde{r}$.
Furthermore, it can be shown that there is a unique maximal
self-stable family $\tilde{X}_{p_0,\tilde{r}}$ and that
$$
\card\Omega^X_{p_0,\tilde{r}}=4,\quad p_0=(0,0),
$$
where $\Omega^X_{p_0,\tilde{r}}$ is the metric identification of
$\tilde{X}_{p_0,\tilde{r}}$. Equalities \eqref{3:eq3.23} imply
that $\Omega^X_{p_0,\tilde{r}}$ is isometric to the pseudo-linear
quadruples and, consequently, is tangent by Corollary
\ref{3:c3.4}.
\end{example}

\begin{remark}\label{3:r3.8}
Some details dropped under consideration of the above example can
be easily regenerated from \cite{DM2}.
\end{remark}

The next corollary of Theorem \ref{3:T3.2} clarifies "the
geometrical sense"  of the factor $\Phi(x,y,z)$ in limit relation
\eqref{3:eq3.3}.

\begin{corollary}\label{3:c3.9}
Let $(X,d)$ be a metric space and let $p$ be a limit point of $X$.
The following statements are equivalent:
\newline
(i) $(\Omega_{p,\tilde{r}},\rho^{s_1})\in\mathfrak{M}$ for every
pretangent $(\Omega_{p,\tilde{r}},\rho)$ and all
$s_1\in(0,\infty)$;
\newline
(ii) $\card\Omega_{p,\tilde{r}}\leqslant 2$ for every pretangent
$(\Omega_{p,\tilde{r}},\rho)$;
\newline
(iii) $\underset{x,y\rightarrow p}{\lim}F(x,y)=0$ where the
function $F$ was defined by \eqref{2:eq2.13};
\newline
(iv) $\underset{\underset{(x,y,z)\in X^{+3}}{x,y,z\rightarrow
p}}{\lim}\Phi(x,y,z)=0$ where the function $\Phi$ was defined by
\eqref{2:eq2.14}.
\end{corollary}

\begin{proof}
$(i)\Rightarrow(ii)$. It follows from the definition of
$\mathfrak{M}$ and from Remark \ref{2:r2.8}.

$(ii)\Rightarrow(iii)$ was proved in \cite{DM2}.

$(iii)\Rightarrow(iv)$ is immediate from the definitions of the
functions $\Phi$ and $F$.

$(iv)\Rightarrow(i)$. To prove this note that $(iv)$ implies
\eqref{3:eq3.3} because the values of
$$
\frac{\Psi(x,y,z)s^2(x,y,z)}{(s_1-s(x,y,z))^2}
$$
are bounded away from zero. Thus (i) follows by Theorem
\ref{3:T3.2}.
\end{proof}

{\bf Acknowledgment.} The first author was partially supported by
the State Foundation for Basic Researches of Ukraine, Grant $\Phi$
25.1/055.

\newpage
{\bf Oleksiy Dovgoshey}
\\
Institute of Applied Mathematics and Mechanics of NASU, R.Luxemburg
str. 74, Donetsk 83114, Ukraine.
\\
{\bf E-mail:} aleksdov@mail.ru.
\bigskip

{\bf Dmytro Dordovskyi}
\\
Institute of Applied Mathematics and Mechanics of NASU, R.Luxemburg
str. 74, Donetsk 83114, Ukraine.
\\
{\bf E-mail:} dordovskydmitry@gmail.com.
\end{document}